\documentclass{amsart}
\usepackage{amsmath,amssymb,graphicx,latexsym}
\usepackage{hyperref}
\usepackage[T1]{fontenc}
\usepackage{textcomp}
\usepackage[utf8]{inputenc}
\usepackage{mathtools} 
\usepackage{amsfonts} 
\usepackage{times, url}
\usepackage{siunitx}
\usepackage{float}
\usepackage{listings}
\usepackage{subcaption}
\usepackage{nccmath}
\usepackage[english]{babel}
\newtheorem{theorem}{Theorem}
\newtheorem{lemma}[theorem]{Lemma}
\newtheorem{corollary}[theorem]{Corollary}
\newtheorem{conjecture}[theorem]{Conjecture}

\theoremstyle{definition}

\newtheorem*{definition}{Definition}
\newtheorem*{notation*}{Notation}

\begin{document}

\title{On the general Smarandache's sigma product of digits}

\author{Luca Onnis}

\maketitle

\begin{abstract}
    This paper investigates the behaviour of one of the most famous Smarandache's sequence given by A061076 on oeis \cite{1} . In particular we first study the behaviour of two sequences (A061077 \cite{3}, A061078 \cite{4}) strictly connected with the main Smarandache's sigma product of digits. We'll solve some open problems such as the determination of an upper bound for these sequences (which hold for all $n\in\mathbb{N}$) and the determination of a closed formula for each $a(n)$ and $b(n)$. Then combining these results it will be possible to understand the behaviour of the general sequence $c(n)$. Every result will be accompanied by Wolfram Mathematica \cite{5} scripts examples in order to support our thesis.
\end{abstract}
\section{Introduction}\label{Introduction}
Let us give the main definitions and the notations which we'll use during the paper:
\begin{definition}
\begin{enumerate}
    \item Let $n\in\mathbb{N}$. We define $a(n)$ as the $n$th term of the Smarandache's even sequence \cite{2}, or the sum of the products of the digits of the first $n$ even numbers.
    \item Let $n\in\mathbb{N}$. We define $b(n)$ as the $n$th term of the Smarandache's odd sequence \cite{2}, or the sum of the products of the digits of the first $n$ odd numbers.
    \item Let $n\in\mathbb{N}$. We define $c(n)$ as the $n$th term of the general Smarandache's general sequence \cite{2}, or the sum of the products of the digits of the first $n$ numbers.
\end{enumerate}
\end{definition}
The first terms of $a(n)$ are:
\[
0,2,6,12,20,20,22,26,32,40,40,\dots
\]
The first terms of $b(n)$ are:
\[
1, 4, 9, 16, 25, 26, 29, 34, 41, 50,52,\dots
\]
The first terms of $c(n)$ are:
\[
1, 3, 6, 10, 15, 21, 28, 36, 45, 45, 46,\dots
\]
\begin{definition}
\begin{enumerate}
    \item Let $p\in\mathbb{N}$ be an even number. We define the contribution of $p$ in $a(n)$ as the product of its digits. We'll indicate it using the notation $C(p)$.
    \item Let $p\in\mathbb{N}$ be an odd number. We define the contribution of $p$ in $b(n)$ as the product of its digits. We'll indicate it using the same notation $C(p)$.
\end{enumerate}
\end{definition}
For example:
\[
C(3688) = 1152 = 3\cdot6\cdot8\cdot8
\]
\[
C(37) = 21 = 3\cdot 7
\]
\begin{definition}
\begin{enumerate}
    \item Let $b,c\in\mathbb{N}$ even numbers and $a(n)$ be the $n$th term of the Smarandache's even sequence; we define $C(b\rightarrow c)$ as the sum of the "contributes" of all the even numbers from $b$ to $c$. So:
\[
C(b\rightarrow c) = a\Bigl(\frac{c}{2}\Bigr)-a\Bigl(\frac{b}{2}-1\Bigr)
\]
\item Let $d,e\in\mathbb{N}$ even numbers and $b(n)$ be the $n$th term of the Smarandache's odd sequence; we define $C(b\rightarrow c)$ as the sum of the "contributes" of all the odd numbers from $b$ to $c$. So:
\[
C(d\rightarrow e) = b\Bigl(\Bigl\lceil\frac{e}{2}\Bigr\rceil\Bigr)-b\Bigl(\Bigl\lceil\frac{d}{2}\Bigr\rceil-1\Bigr)
\]
\end{enumerate}
\end{definition}
For example:
\[
C(14\rightarrow 22) = 1\cdot4+1\cdot6+1\cdot8+2\cdot0+2\cdot2 = a(11)-a(6)
\]
\[
C(3\rightarrow 7) = 3+5+7 = a(4)-a(1)
\]
\section{Smarandache’s sigma product of digits (even sequence)}
\subsection{First exploration of $a(n)$}
First of all note that:
\[
a(5) = 2+4+6+8+1\cdot 0 = 20 = a(4)
\]
So $a(5)$ is the sum of the contributes of 5 numbers (2,4,6,8,10). But the contribute of a number which contains a 0 in its decimal representation is equal to 0. We can see that:
\begin{table}[H]
    \centering
    \begin{tabular}{c|c|c|c|c}
     $2\rightarrow8$  & $12\rightarrow 98$ & $112\rightarrow 998$ & $1112\rightarrow 9998$ & \dots  \\
     \hline
      20   & 900 & 40500 & 1822500 & \dots 
    \end{tabular}
    \caption{Sum of contributes of numbers from $a$ to $b$ ($a\rightarrow b$) in the sequence $a(n)$}
    \label{tab:my_label}
\end{table}
In general it's possible to prove by induction that:
\[
C(\underbrace{11\dots1}_{\text{"$k$" 1}}2\rightarrow\underbrace{99\dots9}_{\text{"$k$" 9}}8) = 4\cdot 5^{k+1}\cdot 9^{k}
\]
Furthermore, note that:
\[
\begin{cases}
C(0\rightarrow10) = C(2\rightarrow8) \\
C(10\rightarrow100)=C(12\rightarrow98) \\
C(100\rightarrow1000)=C(112\rightarrow998) \\
\mbox{ }\mbox{ }\mbox{ }\mbox{ }\mbox{ }\mbox{ }\mbox{ }\mbox{ }\mbox{ }\vdots\mbox{ }\mbox{ }\mbox{ }\mbox{ }\mbox{ }\mbox{ }\mbox{ }\mbox{ }\mbox{ }\mbox{ }\mbox{ }\mbox{ }\mbox{ }\mbox{ }\mbox{ }\mbox{ }\mbox{ }\mbox{ }\mbox{ }\mbox{ }\mbox{ }\mbox{ }\mbox{ }\vdots
\end{cases}
\]
Because, as said before, the contribute of a number which contains a 0 in its decimal representation is equal to 0. \\
\begin{theorem}\label{tl}
Let $a,b\in\mathbb{N}$ even numbers and $C(a\rightarrow b)$ be the sum of the contributes of the even numbers from $a$ to $b$ in $a(n)$. Then:
\[
C(\underbrace{11\dots1}_{\text{"$k$" 1}}2\rightarrow\underbrace{99\dots9}_{\text{"$k$" 9}}8) = 4\cdot 5^{k+1}\cdot 9^{k}
\]
\end{theorem}
\begin{proof}
As said before we'll prove this result by induction on $k$. The base of our induction argument is $k=0$:
\[
C(2\rightarrow 8) = 20 = 4\cdot 5^{1}\cdot 9^{0}
\]
Because:
\[
C(2\rightarrow 8)=a(4)=2+4+6+8=20
\]
Now suppose that:
\[
C(\underbrace{11\dots1}_{\text{"$k$" 1}}2\rightarrow\underbrace{99\dots9}_{\text{"$k$" 9}}8) = 4\cdot 5^{k+1}\cdot 9^{k}
\]
And we'll prove that:
\[
C(\underbrace{11\dots1}_{\text{"$k+1$" 1}}2\rightarrow\underbrace{99\dots9}_{\text{"$k+1$" 9}}8) = 4\cdot 5^{k+2}\cdot 9^{k+1}
\]
Note that $\underbrace{11\dots1}_{\text{"$k+1$" 1}}2$ is the first number larger than $10^{k+1}$ such that there are no zeros in its base 10 digit representation. Furthermore:
\[
C(\underbrace{11\dots1}_{\text{"$k+1$" 1}}2\rightarrow1\underbrace{99\dots9}_{\text{"$k$" 1}}8) = C(\underbrace{11\dots1}_{\text{"$k$" 1}}2\rightarrow\underbrace{99\dots9}_{\text{"$k$" 9}}8) = 4\cdot 5^{k+1}\cdot 9^{k}
\]
\[
C(2\underbrace{11\dots1}_{\text{"$k$" 1}}2\rightarrow2\underbrace{99\dots9}_{\text{"$k$" 1}}8) = 2\cdot C(\underbrace{11\dots1}_{\text{"$k$" 1}}2\rightarrow\underbrace{99\dots9}_{\text{"$k$" 9}}8) = 8\cdot 5^{k+1}\cdot 9^{k}
\]
\[
C(3\underbrace{11\dots1}_{\text{"$k$" 1}}2\rightarrow3\underbrace{99\dots9}_{\text{"$k$" 1}}8) = 3\cdot C(\underbrace{11\dots1}_{\text{"$k$" 1}}2\rightarrow\underbrace{99\dots9}_{\text{"$k$" 9}}8) = 12\cdot 5^{k+1}\cdot 9^{k}
\]
\[
\vdots
\]
\[
C(9\underbrace{11\dots1}_{\text{"$k$" 1}}2\rightarrow\underbrace{99\dots9}_{\text{"$k+1$" 1}}8) = 9\cdot C(\underbrace{11\dots1}_{\text{"$k$" 1}}2\rightarrow\underbrace{99\dots9}_{\text{"$k$" 9}}8) = 36\cdot 5^{k+1}\cdot 9^{k}
\]
In order to understand this fundamental concept of the proof, consider the following subcase:
\[
C(12\rightarrow 98) = 900 \mbox{ $\wedge$ } C(112\rightarrow 998) = 40500
\]
Then, since the first digit of 112 is 1:
\[
C(112\rightarrow198) = C(12\rightarrow98) = 900
\]
But the first digit of 212 is 2, so:
\[
C(212\rightarrow 298) = 2\cdot C(12\rightarrow98) = 1800
\]
and so on until:
\[
C(912\rightarrow 998) = 9\cdot C(12\rightarrow98) = 8100
\]
So we'll have that:
\[
C(\underbrace{11\dots1}_{\text{"$k+1$" 1}}2\rightarrow\underbrace{99\dots9}_{\text{"$k+1$" 9}}8) = \sum_{i=1}^{9} [4i\cdot 5^{k+1}\cdot 9^{k}] = 4\cdot 5^{k+1}\cdot 9^{k}\cdot\sum_{i=1}^{9}i
\]
And finally:
\[
C(\underbrace{11\dots1}_{\text{"$k+1$" 1}}2\rightarrow\underbrace{99\dots9}_{\text{"$k+1$" 9}}8) = 4\cdot 5^{k+1}\cdot 9^{k}\cdot 45 = 4\cdot 5^{k+2}\cdot 9^{k+1}
\]
\end{proof}
\subsection{A closed formula for $a(n)$}
Let $a_0$ be an even digit and $1\leq a_n\leq 9$, $0\leq a_{n-1},\dots,a_1\leq 9$; it's possible to note the following identities:
\begin{enumerate}
    \item $C(2\rightarrow a_0)$ = $\Bigl[\frac{a_0^{2}}{4}+\frac{a_0}{2}\Bigr]$
    \item $C(2\rightarrow 10a_1+a_0)$ = $20 + a_1\Bigl[\frac{a_0^{2}}{4}+\frac{a_0}{2}\Bigr]+10a_1(a_1-1)$
    \item $C(2\rightarrow 100a_2+10a_1+a_0)$ = $920+a_1a_2\Bigl[\frac{a_0^{2}}{4}+\frac{a_0}{2}\Bigr]+10a_1a_2(a_1-1)+450a_2(a_2-1)$
    \item $C(2\rightarrow 1000a_3+100a_2+10a_1+a_0)$ = $41420 + a_1a_2a_3\Bigl[\frac{a_0^{2}}{4}+\frac{a_0}{2}\Bigr]+
    10a_1a_2a_3(a_1-1)+450a_2a_3(a_2-1)+20250a_3(a_3-1)$
    \item\dots
\end{enumerate}
For example:
\[
a(3567)=C(2\rightarrow7134) = 41420 + 21\cdot 6+ 10\cdot 21 \cdot (3-1) + 450\cdot 7 \cdot (1-1)+20250\cdot 7 \cdot (7-1)
\]
\[
a(3567)=892466
\]
As you can see using this code:
\begin{lstlisting}[language=Mathematica,caption={To compute the 3567-th term}]
Accumulate[Times @@@ IntegerDigits[Range[2, 10000, 2]]][[3567]]
\end{lstlisting}
The following theorem will generalize this recurrence.
\begin{theorem}\label{t1}
Let $a(n)$ be the sum of the products of the digits of the first $n$ even numbers. Then:
\[
a(n) = C(2\rightarrow2n) = C\Bigl(2\rightarrow\sum_{\substack{k=0\\2\mid a_0}}^{m}a_k\cdot 10^k\Bigr)
\]
\[
a(n) = \frac{5}{11}(45^{m}-1)+\sum_{k=1}^{m}\Bigl[\prod_{k<j\leq m}a_j\Bigr]\Bigl[2\cdot 5^{k}\cdot 9^{k-1}\cdot a_k \cdot (a_k-1)\Bigr]+\Bigl[\prod_{i=1}^{m}a_i\Bigr]\Bigl[\frac{a_0^{2}}{4}+\frac{a_0}{2}\Bigr]
\]
\end{theorem}
\begin{proof}
We'll prove this result by induction on $m$ (which is the number of digits of $2n$ in its decimal representation minus 1). From the definition of this sequence, $a(n)$ is equal to the sum of the contributes of every even number from $2$ to $2n$. The base case is when $m=0$, so when there is only one (even) digit. 
\[
2n = a_0\in\{2,4,6,8\} \implies n\in\{1,2,3,4\}
\]
and:
\[
a(n) = C(2\rightarrow a_0) = \begin{cases}
2 \mbox{ if $a_0=2$} \\ 6 \mbox{ if $a_0=4$} \\ 12 \mbox{ if $a_0=6$} \\ 20 \mbox{ if $a_0=8$}
\end{cases}
\]
so:
\[
C(2\rightarrow a_0) = \Bigl[\frac{a_0^{2}}{4}+\frac{a_0}{2}\Bigr]
\]
Suppose that the identity holds for $m$ and we'll prove that:
\[
C(2\rightarrow2n) = \frac{5}{11}(45^{m+1}-1)+\sum_{k=1}^{m+1}\Bigl[\prod_{k<j\leq m+1}a_j\Bigr]\Bigl[2\cdot 5^{k}\cdot 9^{k-1}\cdot a_k \cdot (a_k-1)\Bigr]+\Bigl[\prod_{i=1}^{m+1}a_i\Bigr]\Bigl[\frac{a_0^{2}}{4}+\frac{a_0}{2}\Bigr]
\]
where $2n$ has $m+2$ digits in base 10 and its representation is:
\[
2n = \sum_{\substack{k=0\\2\mid a_0}}^{m+1}a_k\cdot 10^k = \underbrace{\sum_{\substack{k=0\\2\mid a_0}}^{m}a_k\cdot10^k}_{\text{$2n_0$}} +a_{m+1}10^{m+1}
\]
So:
\[
C(2\rightarrow2n)=C(2\rightarrow10^{m+1})+C(10^{m+1}\rightarrow2n) = C(2\rightarrow\underbrace{99\dots9}_{\text{"$m$" 9}}8)+C(\underbrace{11\dots1}_{\text{"$m+1$" 1}}2\rightarrow 2n)
\]
Because the contribute of numbers which cointais $0$ in their decimal representation is equal to $0$. Furthermore:
\begin{equation}\label{eq4}
    C(2\rightarrow\underbrace{99\dots9}_{\text{"$m$" 9}}8) = \frac{5}{11}(45^{m+1}-1)
\end{equation}
And:
\[
C(\underbrace{11\dots1}_{\text{"$m+1$" 1}}2\rightarrow 2n) = C(\underbrace{11\dots1}_{\text{"$m$" 1}}2\rightarrow 1\underbrace{99\dots9}_{\text{"$m-1$" 9}}8)+ 2C(\underbrace{11\dots1}_{\text{"$m$" 1}}2\rightarrow 1\underbrace{99\dots9}_{\text{"$m-1$" 9}}8)
\]
\[
+\dots+(a_{m+1}-1)C(\underbrace{11\dots1}_{\text{"$m$" 1}}2\rightarrow 1\underbrace{99\dots9}_{\text{"$m-1$" 9}}8)+a_{m+1}C(\underbrace{11\dots1}_{\text{"$m$" 1}}2\rightarrow 2n_0)
\]
So:
\[
C(\underbrace{11\dots1}_{\text{"$m+1$" 1}}2\rightarrow 2n) = \frac{a_{m+1}(a_{m+1}-1)}{2}C(\underbrace{11\dots1}_{\text{"$m$" 1}}2\rightarrow 1\underbrace{99\dots9}_{\text{"$m-1$" 9}}8)+a_{m+1}C(\underbrace{11\dots1}_{\text{"$m$" 1}}2\rightarrow 2n_0)
\]
But we proved before that:
\begin{equation}\label{eq5}
    C(\underbrace{11\dots1}_{\text{"$m$" 1}}2\rightarrow 1\underbrace{99\dots9}_{\text{"$m-1$" 9}}8)=4\cdot 5^{m+1}\cdot 9^{m}
\end{equation}
\[
C(\underbrace{11\dots1}_{\text{"$m+1$" 1}}2\rightarrow 2n) = \frac{a_{m+1}(a_{m+1}-1)}{2}4\cdot 5^{m+1}\cdot 9^{m}+a_{m+1}C(\underbrace{11\dots1}_{\text{"$m$" 1}}2\rightarrow 2n_0)
\]
\begin{equation}\label{eq6}
    C(\underbrace{11\dots1}_{\text{"$m+1$" 1}}2\rightarrow 2n) = a_{m+1}(a_{m+1}-1)\cdot 2\cdot 5^{m+1}\cdot 9^{m}+a_{m+1}C(\underbrace{11\dots1}_{\text{"$m$" 1}}2\rightarrow 2n_0)
\end{equation}
But from the induction hypothesis, since $2n_0$ has $m+1$ digits:
\[
C(\underbrace{11\dots1}_{\text{"$m$" 1}}2\rightarrow 2n_0) = C(2\rightarrow 2n_0) - C(2\rightarrow \underbrace{99\dots9}_{\text{"$m$" 9}}8)
\]
\[
= \frac{5}{11}(45^{m}-1)+\sum_{k=1}^{m}\Bigl[\prod_{k<j\leq m}a_j\Bigr]\Bigl[2\cdot 5^{k}\cdot 9^{k-1}\cdot a_k \cdot (a_k-1)\Bigr]+\Bigl[\prod_{i=1}^{m}a_i\Bigr]\Bigl[\frac{a_0^{2}}{4}+\frac{a_0}{2}\Bigr] - \frac{5}{11}(45^{m}-1)
\]
\[
= \sum_{k=1}^{m}\Bigl[\prod_{k<j\leq m}a_j\Bigr]\Bigl[2\cdot 5^{k}\cdot 9^{k-1}\cdot a_k \cdot (a_k-1)\Bigr]+\Bigl[\prod_{i=1}^{m}a_i\Bigr]\Bigl[\frac{a_0^{2}}{4}+\frac{a_0}{2}\Bigr] = S
\]
And finally combining together equations \ref{eq4},\ref{eq5},\ref{eq6}:
\[
C(2\rightarrow 2n) = \frac{5}{11}(45^{m+1}-1) + a_{m+1}(a_{m+1}-1)\cdot 2\cdot 5^{m+1}\cdot 9^{m} + S
\]
But $a_{m+1}(a_{m+1}-1)\cdot 2\cdot 5^{m+1}\cdot 9^{m}$ represents the $m+1$ term of $S$. So:
\[
C(2\rightarrow 2n) = \frac{5}{11}(45^{m+1}-1)+\sum_{k=1}^{m+1}\Bigl[\prod_{k<j\leq m+1}a_j\Bigr]\Bigl[2\cdot 5^{k}\cdot 9^{k-1}\cdot a_k \cdot (a_k-1)\Bigr]+\Bigl[\prod_{i=1}^{m+1}a_i\Bigr]\Bigl[\frac{a_0^{2}}{4}+\frac{a_0}{2}\Bigr]
\]
which is our thesis.
\end{proof}
This formula could be implemented in Mathematica \cite{5} using this code:
\begin{lstlisting}[language=Mathematica,caption={To compute the first 10000 terms}]
Table[5/11*(45^(Length[IntegerDigits[2n]]-1)-1) + 
Sum[(Product[
IntegerDigits[2n][[j]], {j, 1, 
Length[IntegerDigits[2n]]-k-1}])*(2*(5^
k)*(9^(k-1))*(IntegerDigits[
2n][[Length[IntegerDigits[2 n]]-k]])*(IntegerDigits[
2n][[Length[IntegerDigits[2 n]]-k]] - 1)), {k, 1, 
Length[IntegerDigits[2n]] - 1}] + (Product[
IntegerDigits[2n][[i]], {i,1, Length[IntegerDigits[2n]] - 
1}]) (((IntegerDigits[2 n][[Length[IntegerDigits[2 n]]]])^2)/4 + 
(IntegerDigits[2 n][[Length[IntegerDigits[2 n]]]])/2), 
{n, 1, 10000}]
\end{lstlisting}
Which is equivalent to:
\begin{lstlisting}[language=Mathematica,caption={To compute the first 10000 terms}]
Accumulate[Times @@@ IntegerDigits[Range[2, 20000, 2]]]
\end{lstlisting}
\subsection{Upper bound for $a(n)$}
We'll prove an important inequality between the $n$th term of $a(n)$ and a function which depends on $n$. Look at the following graphs:
\begin{figure}[H]
\centering
\begin{subfigure}[b]{0.4\linewidth}
    \includegraphics[width=\linewidth]{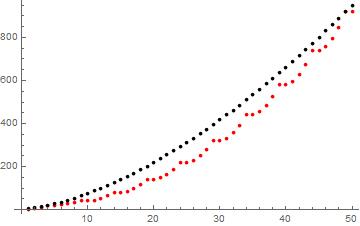}
    \caption{Graph of $a(n)$ (red) and $f(n)$ (black) where $n\in[1,50]$}
  \end{subfigure}
\begin{subfigure}[b]{0.4\linewidth}
    \includegraphics[width=\linewidth]{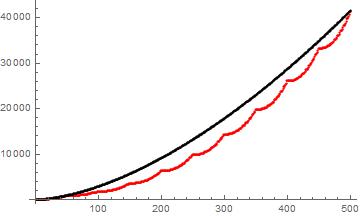}
    \caption{Graph of $a(n)$ (red) and $f(n)$ (black) where $n\in[1,500]$}
  \end{subfigure}
  \caption{Comparison between $a(n)$ and $f(n)$ in different intervals}
  \label{fig:coffee}
\end{figure}
\begin{figure}[H]
\centering
\begin{subfigure}[b]{0.4\linewidth}
    \includegraphics[width=\linewidth]{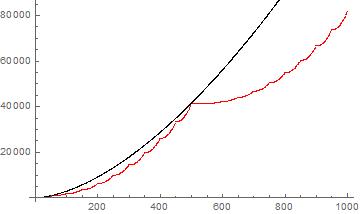}
    \caption{Graph of $a(n)$ (red) and $f(n)$ (black) where $n\in[1,1000]$}
  \end{subfigure}
\begin{subfigure}[b]{0.4\linewidth}
    \includegraphics[width=\linewidth]{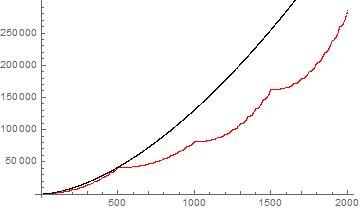}
    \caption{Graph of $a(n)$ (red) and $f(n)$ (black) where $n\in[1,2000]$}
  \end{subfigure}
  \caption{Comparison between $a(n)$ and $f(n)$ in different intervals}
  \label{fig:coffee}
\end{figure}

\begin{theorem}\label{a(n)3}
Let $a(n)$ be the sum of the products of the digits of the first $n$ even numbers. Then:
\[
a(n) \leq \frac{5}{11}\Bigl[45^{\log_{10}(\frac{n+1}{5})+1}-1\Bigr] = f(n) \mbox{ $\forall n\in\mathbb{N}$}
\]
And we have the equality if and only if $n=5\cdot 10^{k}-1$ for some $k\in\mathbb{N}$.
\end{theorem}
\begin{proof}
First of all note that:
\[
\frac{5}{11}\Bigl[45^{\log_{10}(\frac{n+1}{5})+1}-1\Bigr]\in\mathbb{N}\Longleftrightarrow n = 5\cdot 10^{k}-1
\]
But furthermore $a(5\cdot 10^{k}-1)=a(5\cdot 10^k)$, $\forall k\in\mathbb{N}$. In fact:
\[
a(5\cdot 10^{k}-1) = C(2\rightarrow 10^{k+1}-2) = C(2\rightarrow 10^{k+1}) = a(5\cdot 10^{k})
\]
But in theorem 2 we derived a closed formula for $a(n)$ which depends on the digits in the decimal representation of $2n$. So:
\[
a(5\cdot 10^{k}-1)=a(5\cdot 10^{k})=C(2\rightarrow 1\underbrace{00\dots0}_{\text{$k+1$ zeros}})=\frac{5}{11}(45^{k+1}-1) = f(5\cdot10^{k}-1)
\]
So $a(n)=f(n)$ if and only if $n=5\cdot 10^{k}-1$ for some $k\in\mathbb{N}$. \\
Now we want to prove that:
\[
\frac{5}{11}(45^{m}-1)+\sum_{k=1}^{m}\Bigl[\prod_{k<j\leq m}a_j\Bigr]\Bigl[2\cdot 5^{k}\cdot 9^{k-1}\cdot a_k \cdot (a_k-1)\Bigr]+\Bigl[\prod_{i=1}^{m}a_i\Bigr]\Bigl[\frac{a_0^{2}}{4}+\frac{a_0}{2}\Bigr] \leq \frac{5}{11}\Bigl[45^{\log_{10}(\frac{n+1}{5})+1}-1\Bigr]
\]
$\forall n\in\mathbb{N}$. 
Note that here $m+1$ is the number of digits of $2n$. \\
We'll prove this fact by contradiction. We know that $a(n)$ and $f(n)$ are monotone increasing functions and $a(n)=f(n)$ if and only if $n=5\cdot 10^{k}-1$; so if $\exists n\in\mathbb{N}$ such that $a(n)>f(n)$, then the inequality must hold in a closed interval of the form: $(5\cdot 10^{k}-1,5\cdot 10^{k+1}-1)\subset\mathbb{N}$. That's because $a(n)$ and $f(n)$ intersect each other only when $n=5\cdot 10^{k}-1$ for some $k\in\mathbb{N}$. So:
\[
a(n)>f(n) \mbox{ $\forall n\in (5\cdot 10^{k}-1,5\cdot 10^{k+1}-1)$}
\]
But now we know that:
\[
a(5\cdot 10^{k}-1)=f(5\cdot 10^{k}-1) \mbox{ and } a(5\cdot 10^{k})>f(5\cdot 10^{k})
\]
And this isn't true because $a(5\cdot 10^{k})=a(5\cdot 10^{k}-1)=f(5\cdot 10^{k})$. We arrived at a contradiction caused by supposing that $a(n)>f(n)$ for some $n\in\mathbb{N}$.
\end{proof}
\section{Smarandache’s sigma product of digits (odd sequence)}
\subsection{First exploraton of $b(n)$}
First of all note that as in the first sequence , the contribute of a number which contains a 0 in its decimal representation is equal to 0. We can see that:
\begin{table}[H]
    \centering
    \begin{tabular}{c|c|c|c|c}
     $1\rightarrow9$  & $11\rightarrow 99$ & $111\rightarrow 999$ & $1111\rightarrow 9999$ & \dots  \\
     \hline
      25  & 1125 & 50625 & 2278125 & \dots 
    \end{tabular}
    \caption{Sum of contributes of numbers from $a$ to $b$ ($a\rightarrow b$) in the sequence $b(n)$}
    \label{tab:my_label}
\end{table}
In general it's possible to prove by induction (using the same technique of the proof of theorem \ref{tl}) that:
\begin{equation}\label{b(n)1}
    C(\underbrace{11\dots1}_{\text{"$k$" 1}}\rightarrow\underbrace{99\dots9}_{\text{"$k$" 9}}) = 5^{k+1}\cdot9^{k-1}
\end{equation}
Furthermore, note that:
\[
\begin{cases}
C(1\rightarrow9) =C(1\rightarrow9) \\
C(11\rightarrow99)=C(11\rightarrow99) \\
C(101\rightarrow999)=C(111\rightarrow999) \\
\mbox{ }\mbox{ }\mbox{ }\mbox{ }\mbox{ }\mbox{ }\mbox{ }\mbox{ }\mbox{ }\vdots\mbox{ }\mbox{ }\mbox{ }\mbox{ }\mbox{ }\mbox{ }\mbox{ }\mbox{ }\mbox{ }\mbox{ }\mbox{ }\mbox{ }\mbox{ }\mbox{ }\mbox{ }\mbox{ }\mbox{ }\mbox{ }\mbox{ }\mbox{ }\mbox{ }\mbox{ }\mbox{ }\vdots
\end{cases}
\]
Because, as said before, the contribute of a number which contains a 0 in its decimal representation is equal to 0. \\
\subsection{A closed formula for $b(n)$}
Let $a_0$ be an odd digit and $1\leq a_n\leq 9$, $0\leq a_{n-1},\dots,a_1\leq 9$; it's possible to note the following identities:
\begin{enumerate}
    \item $C(2\rightarrow a_0)$ = $(\frac{a_0+1}{2})^{2}$
    \item $C(2\rightarrow 10a_1+a_0)$ = $25 + a_1(\frac{a_0+1}{2})^{2}+25\frac{a_1(a_1-1)}{2}$
    \item $C(2\rightarrow 100a_2+10a_1+a_0)$ = $1150+a_1a_2(\frac{a_0+1}{2})^{2}+25a_2\frac{a_1(a_1-1)}{2}+1125\frac{a_2(a_2-1)}{2}$
    \item $C(2\rightarrow 1000a_3+100a_2+10a_1+a_0)$ = $51775 +a_1a_2a_3(\frac{a_0+1}{2})^{2}+25a_2a_3\frac{a_1(a_1-1)}{2}+1125a_3\frac{a_2(a_2-1)}{2}+50625\frac{a_3(a_3-1)}{2}$
    \item\dots
\end{enumerate}
For example:
\[
b(4637)=C(1\rightarrow9273) = 51775 + 504 +9450+10125+1822500 = 1894354
\]
\begin{theorem}\label{t2}
Let $b(n)$ be the sum of the products of the digits of the first $n$ even numbers. Then:
\[
b(n) = C(1\rightarrow2n-1) = C\Bigl(1\rightarrow\sum_{\substack{k=0\\a_0\equiv_2 1}}^{m}a_k\cdot 10^k\Bigr)
\]
\[
b(n) = \frac{25}{44}\Bigl[45^{m}-1\Bigr]+\sum_{k=1}^{m}\Bigl[\prod_{k<j\leq m}a_j\Bigr]\Bigl[5^{k+1}\cdot 9^{k-1}\cdot \frac{a_k \cdot (a_k-1)}{2}\Bigr]+\Bigl[\prod_{i=1}^{m}a_i\Bigr]\Bigl[\Bigl(\frac{a_0+1}{2}\Bigr)^{2}\Bigr]
\]
\end{theorem} 
\begin{proof}
We'll prove this result by induction on $m$ (which is the number of digits of $2n-1$ in its decimal representation minus 1). From the definition of this sequence, $b(n)$ is equal to the sum of the contributes of every even number from $1$ to $2n-1$. The base case is when $m=0$, so when there is only one (odd) digit. 
\[
2n-1 = a_0\in\{1,3,5,7,9\} \implies n\in\{1,2,3,4,5\}
\]
and:
\[
b(n) = C(1\rightarrow a_0) = \begin{cases}
1 \mbox{ if $a_0=1$} \\ 4 \mbox{ if $a_0=3$} \\ 9\mbox{ if $a_0=5$} \\ 16 \mbox{ if $a_0=7$} \\ 25 \mbox{ if $a_0=9$}
\end{cases}
\]
so:
\[
C(1\rightarrow a_0) =\Bigl(\frac{a_0+1}{2}\Bigr)^{2}
\]
Suppose that the identity holds for $m$ and we'll prove that:
\[
C(1\rightarrow2n-1) = \frac{25}{44}\Bigl[45^{m+1}-1\Bigr]+\sum_{k=1}^{m+1}\Bigl[\prod_{k<j\leq m+1}a_j\Bigr]\Bigl[5^{k+1}\cdot 9^{k-1}\cdot \frac{a_k \cdot (a_k-1)}{2}\Bigr]+\Bigl[\prod_{i=1}^{m+1}a_i\Bigr]\Bigl[\Bigl(\frac{a_0+1}{2}\Bigr)^{2}\Bigr]
\]
where $2n-1$ has $m+2$ digits in base 10 and its representation is:
\[
2n-1 = \sum_{\substack{k=0\\a_0\equiv_2 1}}^{m+1}a_k\cdot 10^k = \underbrace{\sum_{\substack{k=0\\a_0\equiv_2 1}}^{m}a_k\cdot10^k}_{\text{$2n_0-1$}} +a_{m+1}10^{m+1}
\]
So:
\[
C(1\rightarrow2n-1)=C(1\rightarrow10^{m+1})+C(10^{m+1}\rightarrow2n-1) = C(1\rightarrow\underbrace{99\dots9}_{\text{"$m+1$" 9}})+C(\underbrace{11\dots1}_{\text{"$m+2$" 1}}\rightarrow 2n-1)
\]
Because the contribute of numbers which cointais $0$ in their decimal representation is equal to $0$. Furthermore:
\begin{equation}\label{eq1}
    C(1\rightarrow\underbrace{99\dots9}_{\text{"$m+1$" 9}}) = \frac{25}{44}(45^{m+1}-1)
\end{equation}
And:
\[
C(\underbrace{11\dots1}_{\text{"$m+2$" 1}}\rightarrow 2n-1) = C(\underbrace{11\dots1}_{\text{"$m+1$" 1}}\rightarrow 1\underbrace{99\dots9}_{\text{"$m$" 9}})+ 2C(\underbrace{11\dots1}_{\text{"$m+1$" 1}}\rightarrow 1\underbrace{99\dots9}_{\text{"$m$" 9}})
\]
\[
+\dots+(a_{m+1}-1)C(\underbrace{11\dots1}_{\text{"$m+1$" 1}}\rightarrow 1\underbrace{99\dots9}_{\text{"$m$" 9}})+a_{m+1}C(\underbrace{11\dots1}_{\text{"$m+1$" 1}}\rightarrow 2n_0-1)
\]
So:
\begin{equation}\label{eq2}
C(\underbrace{11\dots1}_{\text{"$m+2$" 1}}\rightarrow 2n-1) = \frac{a_{m+1}(a_{m+1}-1)}{2}C(\underbrace{11\dots1}_{\text{"$m+1$" 1}}\rightarrow 1\underbrace{99\dots9}_{\text{"$m$" 9}})+a_{m+1}C(\underbrace{11\dots1}_{\text{"$m+1$" 1}}\rightarrow 2n_0-1)    
\end{equation}
But we proved before that: 
\begin{equation}\label{eq3}
    C(\underbrace{11\dots1}_{\text{"$m+1$" 1}}\rightarrow 1\underbrace{99\dots9}_{\text{"$m$" 9}})=5^{m+2}\cdot9^{m}
\end{equation}
\[
C(\underbrace{11\dots1}_{\text{"$m+2$" 1}}\rightarrow 2n-1) = \frac{a_{m+1}(a_{m+1}-1)}{2}\cdot 5^{m+2}\cdot9^{m}+a_{m+1}C(\underbrace{11\dots1}_{\text{"$m+1$" 1}}\rightarrow 2n_0-1)
\]
But from the induction hypothesis, since $2n_0-1$ has $m+1$ digits:
\[
C(\underbrace{11\dots1}_{\text{"$m+1$" 1}}\rightarrow 2n_0-1) = C(1\rightarrow 2n_0-1) - C(1\rightarrow \underbrace{99\dots9}_{\text{"$m$" 9}})
\]
\[
= \frac{25}{44}\Bigl[45^{m}-1\Bigr]+\sum_{k=1}^{m}\Bigl[\prod_{k<j\leq m}a_j\Bigr]\Bigl[5^{k+1}\cdot 9^{k-1}\cdot \frac{a_k \cdot (a_k-1)}{2}\Bigr]+\Bigl[\prod_{i=1}^{m}a_i\Bigr]\Bigl[\Bigl(\frac{a_0+1}{2}\Bigr)^{2}\Bigr] -\frac{25}{44}\Bigl[45^{m}-1\Bigr]
\]
\[
= \sum_{k=1}^{m}\Bigl[\prod_{k<j\leq m}a_j\Bigr]\Bigl[5^{k+1}\cdot 9^{k-1}\cdot \frac{a_k \cdot (a_k-1)}{2}\Bigr]+\Bigl[\prod_{i=1}^{m}a_i\Bigr]\Bigl[\Bigl(\frac{a_0+1}{2}\Bigr)^{2}\Bigr] = S
\]
And finally combining equations \ref{eq1}, \ref{eq2}, \ref{eq3}:
\[
C(1\rightarrow 2n-1) =  \frac{25}{44}(45^{m+1}-1) + \frac{a_{m+1}(a_{m+1}-1)}{2}5^{m+2}\cdot9^{m} + S
\]
But $\frac{a_{m+1}(a_{m+1}-1)}{2}5^{m+2}\cdot9^{m}$ represents the $m+1$ term of $S$. So:
\[
C(1\rightarrow 2n-1) = \frac{25}{44}\Bigl[45^{m+1}-1\Bigr]+\sum_{k=1}^{m+1}\Bigl[\prod_{k<j\leq m+1}a_j\Bigr]\Bigl[5^{k+1}\cdot 9^{k-1}\cdot \frac{a_k \cdot (a_k-1)}{2}\Bigr]+\Bigl[\prod_{i=1}^{m+1}a_i\Bigr]\Bigl[\Bigl(\frac{a_0+1}{2}\Bigr)^{2}\Bigr]
\]
which is our thesis.
\end{proof}
\subsection{Upper bound for $b(n)$}
As for the sequence $a(n)$ we'll prove an important inequality between the $n$th term of $b(n)$ and a function which depends on $n$. Look at the following graphs:
\begin{figure}[H]
\centering
\begin{subfigure}[b]{0.4\linewidth}
    \includegraphics[width=\linewidth]{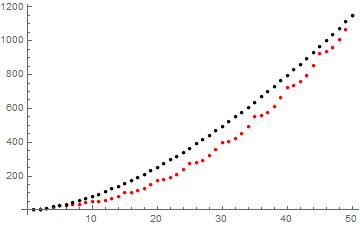}
    \caption{Graph of $b(n)$ (red) and $g(n)$ (black) where $n\in[1,50]$}
  \end{subfigure}
\begin{subfigure}[b]{0.4\linewidth}
    \includegraphics[width=\linewidth]{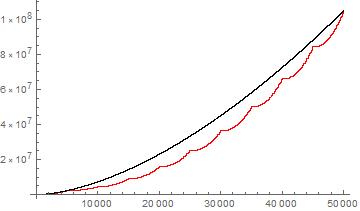}
    \caption{Graph of $b(n)$ (red) and $g(n)$ (black) where $n\in[1,50000]$}
  \end{subfigure}
  \caption{Comparison between $b(n)$ and $g(n)$ in different intervals}
  \label{fig:coffee}
\end{figure}
\begin{theorem}\label{b(n)2}
Let $b(n)$ be the sum of the products of the digits of the first $n$ odd numbers. Then:
\[
b(n) \leq\frac{25}{44}\Bigl[45^{\log_{10}(\frac{n}{5})+1}-1\Bigr] = g(n) \mbox{ $\forall n\in\mathbb{N}$}
\]
And we have the equality if and only if $n=5\cdot 10^{k}$ for some $k\in\mathbb{N}$.
\end{theorem}
\begin{proof}
First of all note that:
\[
\frac{25}{44}\Bigl[45^{\log_{10}(\frac{n}{5})+1}-1\Bigr]\in\mathbb{N}\Longleftrightarrow n = 5\cdot 10^{k}
\]
Furthermore:
\[
b(5\cdot 10^{k}) = C(1\rightarrow 10^{k+1}-1) = C(1\rightarrow\underbrace{99\dots 9}_{\text{"$k+1$" 9}})
\]
But from Theorem 4 we know that:
\[
C(1\rightarrow\underbrace{99\dots 9}_{\text{"$k+1$" 9}}) = \frac{25}{44}\Bigl[45^{k+1}-1\Bigr]
\]
So $b(n)=g(n)$ if and only if $n=5\cdot 10^{k}$ for some $k\in\mathbb{N}$. \\
As in theorem 3, in order to prove the second part we proceed by contradiction; We know that $b(n)$ and $g(n)$ are monotone increasing functions and $b(n)=g(n)$ if and only if $n=5\cdot 10^{k}$; so if $\exists n\in\mathbb{N}$ such that $b(n)>g(n)$, then the inequality must hold in an open interval of the form: $(5\cdot 10^{k},5\cdot 10^{k+1})\subset\mathbb{N}$. That's because $b(n)$ and $g(n)$ intersect each other only when $n=5\cdot 10^{k}$ for some $k\in\mathbb{N}$. So:
\[
a(n)>f(n) \mbox{ $\forall n\in (5\cdot 10^{k},5\cdot 10^{k+1})$}
\]
But now we know that:
\[
b(5\cdot 10^{k})=g(5\cdot 10^{k}) \mbox{ and } b(5\cdot 10^{k}+1)>g(5\cdot 10^{k}+1)
\]
And this isn't true because $b(5\cdot 10^{k}+1)=b(5\cdot 10^{k})=g(5\cdot 10^{k})$. We arrived at a contradiction caused by supposing that $b(n)>g(n)$ for some $n\in\mathbb{N}$.
\end{proof}
\section{The general Smarandache’s sigma product of digits}
\subsection{First exploration of $c(n)$}
First of all note that:
\begin{equation}\label{c(n)1}
  c(2n) = a(n)+b(n) \mbox{ $\forall n\in\mathbb{N}$}  
\end{equation}
In fact $a(n)$ gives the sum of the contributes of the even numbers less than or equal to $n$ while $b(n)$ gives the contributes of the odd numbers less than or equal to $n$.
\subsection{A closed formula for $c(n)$}
From theorems \ref{t1}, \ref{t2} and from equation \ref{c(n)1} we can easily compute the $n$th term of the general sequence $c(n)$. In fact:
\[
\begin{cases}
c(2n) = a(n)+b(n) \\ c(2n+1) = a(n)+b(n)+C(2n+1)
\end{cases}
\]
Where $C(2n+1)$ is the contribute of $2n+1$ in $c(n)$ (or simply the product of the digits of $2n+1$.
\subsection{Upper bound for $c(n)$}
\begin{lemma}
Let $c(n)$ be the sum of the products of the digits of the first $n$ numbers and $C(n)$ the product of the digits of $n$. Then:
\[
c(n) \leq \begin{cases}
\frac{25}{44}\Bigl[45^{\log_{10}(n)}-1\Bigr]+\frac{5}{11}\Bigl[45^{\log_{10}(n+2)}-1\Bigr] \mbox{ if $n$ is even} \\
\frac{25}{44}\Bigl[45^{\log_{10}(n)}-1\Bigr]+\frac{5}{11}\Bigl[45^{\log_{10}(n+2)}-1\Bigr] + C(n) \mbox{ if $n$ is odd}
\end{cases}
\]
$\forall n\in\mathbb{N}$
\end{lemma}
\begin{proof}
The result is simply obtained by combining the previous theorems regarding the upper bounds of $a(n)$ and $b(n)$. If $n$ is an even number, then $n=2n_1$ for some $n_1\in\mathbb{N}$. Furthermore , from equation \ref{c(n)1} we know that:
\[
c(2n_1) = a(n_1)+b(n_1) = a\Bigl(\frac{n}{2}\Bigr)+b\Bigl(\frac{n}{2}\Bigr)
\]
But combining theorem  \ref{a(n)3} and theorem \ref{b(n)2} we'll have the following inequality:
\[
c(n) = a\Bigl(\frac{n}{2}\Bigr)+b\Bigl(\frac{n}{2}\Bigr) \leq \frac{25}{44}\Bigl[45^{\log_{10}(n)}-1\Bigr]+\frac{5}{11}\Bigl[45^{\log_{10}(\frac{n+2}{10})+1}-1\Bigr]
\]
Instead for odd numbers we'll have the same inequality except that we must add the contribute of the argument $C(n)$ (or simply the product of its digits). In fact if $n$ is odd, then $n=2n_2+1$ for some $n_2\in\mathbb{N}$. So:
\[
c(n)=c(2n_2+1)=c(2n_2)+C(2n_2)=a(n_2)+b(n_2)+C(2n_2) 
\]
And finally:
\[
c(n)\leq \frac{25}{44}\Bigl[45^{\log_{10}(n)}-1\Bigr]+\frac{5}{11}\Bigl[45^{\log_{10}(\frac{n+2}{10})+1}-1\Bigr] + C(n)
\]
\end{proof}
\begin{conjecture}
Let $c(n)$ be the sum of the products of the digits of the first $n$ numbers. Then:
\[
c(n) \leq\frac{25}{44}\Bigl[45^{\log_{10}(n)}-1\Bigr]+\frac{5}{11}\Bigl[45^{\log_{10}(\frac{n+2}{10})+1}-1\Bigr]=h(n) \mbox{ $\forall n\in\mathbb{N}\setminus\{10^{k}-1\}_{k\in\mathbb{Z}^{+}}$} 
\]
\end{conjecture}
We really think that the equality above holds without the $C(n)$ part for all natural numbers not equal to $10^{k}-1$ for some positive integer $k$. Look at the following graphs:
\begin{figure}[H]
\centering
\begin{subfigure}[b]{0.4\linewidth}
    \includegraphics[width=\linewidth]{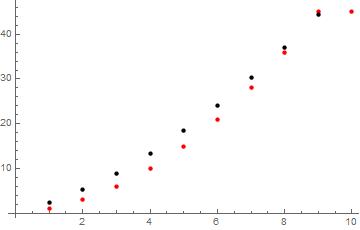}
    \caption{Graph of $c(n)$ (red) and $h(n)$ (black) where $n\in[1,10]$}
  \end{subfigure}
\begin{subfigure}[b]{0.4\linewidth}
    \includegraphics[width=\linewidth]{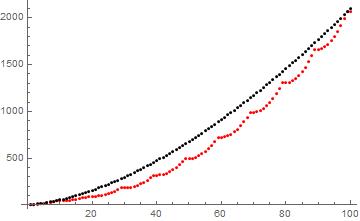}
    \caption{Graph of $c(n)$ (red) and $h(n)$ (black) where $n\in[1,100]$}
  \end{subfigure}
  \caption{Comparison between $c(n)$ and $h(n)$ in different intervals}
  \label{fig:coffee}
\end{figure}
\begin{figure}[H]
\centering
\begin{subfigure}[b]{0.4\linewidth}
    \includegraphics[width=\linewidth]{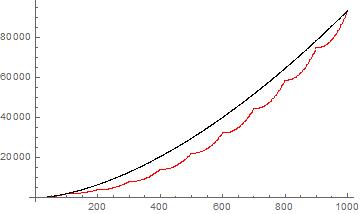}
    \caption{Graph of $c(n)$ (red) and $h(n)$ (black) where $n\in[1,1000]$}
  \end{subfigure}
\begin{subfigure}[b]{0.4\linewidth}
    \includegraphics[width=\linewidth]{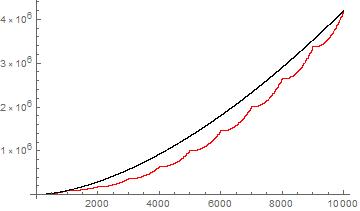}
    \caption{Graph of $c(n)$ (red) and $h(n)$ (black) where $n\in[1,10000]$}
  \end{subfigure}
  \caption{Comparison between $c(n)$ and $h(n)$ in different intervals}
  \label{fig:coffee}
\end{figure}
As you can see $c(n)\leq h(n)$ except for some particular values. The only ones which I found are in fact of the form $10^{k}-1$. Look at the following table:
\begin{table}[H]
    \centering
    \begin{tabular}{c|c|c|c}
    \hline
       $c(9) = 45$  &  $c(99)=2070$ & $c(999) = 93195$ & \dots\\
         \hline
       $h(9) \approx 44.4$  & $h(99) \approx 2066.3$ & $h(999)=93177.9$ & \dots \\
       \hline
    \end{tabular}
    \caption{Comparison between $c(n)$ and $h(n)$ for particular values of $n$}
    \label{tab:my_label}
\end{table}
\section{Final considerations}
In this section we'll analyze the obtained results and we'll combine them together in order to prove some interesting corollaries. In particular we want to study the behaviour of the sequence $\frac{a(n)}{b(n)}$. Surprisingly this sequence is bounded as suggested from these plots:
\begin{figure}[H]
\centering
\begin{subfigure}[b]{0.4\linewidth}
    \includegraphics[width=\linewidth]{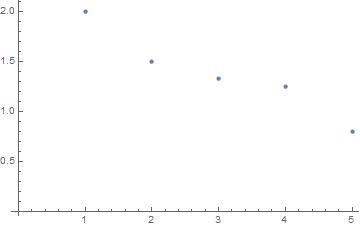}
    \caption{Plot of $\frac{a(n)}{b(n)}$ where $n\in[1,5]$}
  \end{subfigure}
\begin{subfigure}[b]{0.4\linewidth}
    \includegraphics[width=\linewidth]{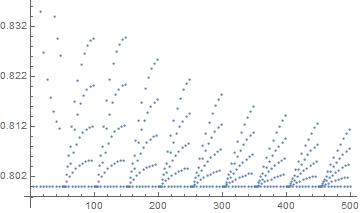}
    \caption{Plot of $\frac{a(n)}{b(n)}$ where $n\in[1,500]$}
  \end{subfigure}
  \caption{Plot of $\frac{a(n)}{b(n)}$ in different intervals}
  \label{fig:coffee}
\end{figure}
\begin{conjecture}\label{con1}
Let $a(n)$ and $b(n)$ be defined as before. Then:
\[
\frac{4}{5}\leq\frac{a(n)}{b(n)}\leq 2\mbox{ $\forall n\in\mathbb{N}$}
\]
\[
\frac{4}{5}\leq\frac{a(n)}{b(n)}< 1 \mbox{ $\forall n\geq5$}
\]
\end{conjecture}
It's easy to see that $\frac{a(n)}{b(n)}< 1$ , $\forall n\geq5$. In fact:
\[
a(4)=2+4+6+8 > 1+3+5+7 = b(4)
\]
But from $n=5$:
\[
a(5) = 2+4+6+8+1\cdot0 < 1+3+5+7+9 = b(5)
\]
Note that the contribution in $a(n)$ of terms like $10,20,30,\dots$ is equal to 0, while in $b(n)$ such numbers can not exists because it counts the contributions only from odd numbers. Using this argument we can see that there are more "zero-contributions" in $a(n)$ than in $b(n)$, and this is sufficient to understand the first inequality. In order to prove the second inequality it's sufficient to show that:
\[
5a(n)-4b(n) \geq 0 \mbox{ $\forall n\in\mathbb{N}$}
\]
The conjecture is probably true as suggested from the plot of $5a(n)-4b(n)$.
\begin{corollary}\label{t8}
Let $a(n)$ and $b(n)$ be defined as before. Then:
\[
\lim_{n\to+\infty}\frac{a(n)}{b(n)}=\frac{4}{5}
\]
\end{corollary}
\begin{proof}
From theorems \ref{a(n)3}, \ref{b(n)2} we know that:
\[
\frac{a(n)}{b(n)}\leq\frac{\frac{5}{11}\Bigl[45^{\log_{10}(\frac{n+1}{5})+1}-1\Bigr]}{\frac{25}{44}\Bigl[45^{\log_{10}(\frac{n}{5})+1}-1\Bigr]}=\frac{4}{5}\underbrace{\frac{\Bigl[45^{\log_{10}(\frac{n+1}{5})+1}-1\Bigr]}{\Bigl[45^{\log_{10}(\frac{n}{5})+1}-1\Bigr]}}_{\text{$r(n)$}}\underset{n\to+\infty}{\longrightarrow}\frac{4}{5}
\]
Hence $\forall\varepsilon>0\exists n_0\in\mathbb{N}$ such that:
\[
|r(n)-1|<\varepsilon \mbox{ $\forall n>n_0$}
\]
So:
\[
\frac{a(n)}{b(n)}-\frac{4}{5}\leq\frac{4}{5}r(n)-\frac{4}{5}=\frac{4}{5}(r(n)-1)<\frac{4}{5}\varepsilon\mbox{ $\forall n>n_0$}
\]
And assuming conjecture \ref{con1}, we'll have:
\[
\frac{a(n)}{b(n)}-\frac{4}{5}\geq0>-\varepsilon\mbox{ $\forall n\in\mathbb{N}$}
\]
Which concludes the proof.
\end{proof}
\section*{Acknowledgment}
I would like to thank the oeis editors and staff for reviewing my changes at these sequences.

\end{document}